\documentclass[11pt, reqno]{amsart}
\usepackage{mathrsfs,amssymb, amscd,amsmath,amsthm}
\usepackage[enableskew,vcentermath]{youngtab}
\usepackage[latin1]{inputenc}
\usepackage{multicol}\multicolsep=0pt
\usepackage{tikz}

\newcommand{\clr}{rgb:black,1;blue,4;red,1}

\newcommand{\bdot}{ node[circle, draw, fill=\clr, thick, inner sep=0pt, minimum width=4pt]{}}

\newcommand{\ob}[1]{\mathsf{#1}}

\newcommand{\B}{\mathcal{B}}

\newcommand{\AB}{\mathcal{AB}}

\newcommand{\lcap}{
\begin{tikzpicture}[baseline = 3pt, scale=0.5, color=\clr]
        \draw[-,thick] (1,0) to[out=up, in=right] (0.53,0.5) to[out=left, in=right] (0.47,0.5);
        \draw[-,thick] (0.49,0.5) to[out=left,in=up] (0,0);
\end{tikzpicture}
}
\newcommand{\lcup}{
\begin{tikzpicture}[baseline = 6pt, scale=0.5, color=\clr]
        \draw[-,thick] (1,1) to[out=down, in=right] (0.53,0.5) to[out=left, in=right] (0.47,0.5);
        \draw[-,thick] (0.49,0.5) to[out=left,in=down] (0,1);
\end{tikzpicture}
}

\newcommand{\swap}{
\begin{tikzpicture}[baseline = 3pt, scale=0.5, color=\clr]
        \draw[-,thick] (0,0) to[out=up, in=down] (1,1);
        \draw[-,thick] (1,0) to[out=up, in=down] (0,1);
\end{tikzpicture}
}

\newcommand{\xdot}{
\begin{tikzpicture}[baseline = 3pt, scale=0.5, color=\clr]

\draw[-,thick] (0,0) to[out=up, in=down] (0,1);
\draw(0,0.5) \bdot;
\end{tikzpicture}
}

 \providecommand{\og}{``}
\providecommand{\fg}{''} \providecommand{\smfandname}{and}

\usepackage{amssymb}
\usepackage{mathrsfs,amssymb}
\usepackage{multicol}\multicolsep=0pt
\usepackage{pstricks,pst-node}
\usepackage[enableskew,vcentermath]{youngtab}

\usepackage[sort]{cite}
\usepackage{xcolor,graphicx}

\def\crulefill{\leavevmode\leaders\hrule height 1pt\hfill\kern 0pt}
\long\def\QUERY#1{%
\leavevmode\newline%
\noindent$\star\star\star$\thinspace\textsf{Comment/Query}\crulefill\newline%
   \space #1\newline\hbox to 120mm{\crulefill}$\star\star\star$\newline}
\newtheorem{Theorem}{Theorem}[section]
\newtheorem{Lemma}[Theorem]{Lemma}

\newtheorem{Prop}[Theorem]{Proposition}

 \theoremstyle{definition}
\newtheorem{example}[Theorem]{Example}

\newtheorem{Defn}[Theorem]{Definition}

\numberwithin{equation}{section}
\theoremstyle{definition}

\makeatletter
\def\enumerate{\begingroup\ifnum\@enumdepth>3\@toodeep\else
      \advance\@enumdepth\@ne
      \edef\@enumctr{enum\romannumeral\the\@enumdepth}%
      \topsep\z@\parskip\z@
      \list{\csname label\@enumctr\endcsname}
        {\@nmbrlisttrue\let\@listctr\@enumctr
         \parsep\z@\itemsep\z@\topsep\z@
         \setcounter{\@enumctr}{0}
         \def\makelabel##1{\hss\llap{\rm ##1}}
       }\fi}

\makeatother

\let\bar=\overline
\let\epsilon=\varepsilon
\def\({\big(}
\def\){\big)}

\def\0{\underline{0}}

\DeclareMathOperator{\End}{End}

\def\t{\mathfrak t}

\def\Hom{\text{Hom}}

{\catcode`\|=\active
  \gdef\set#1{\mathinner{\lbrace\,{\mathcode`\|"8000%
                                   \let|\midvert #1}\,\rbrace}}
  \gdef\seT#1{\mathinner{\Big\lbrace\,{\mathcode`\|"8000%
                                   \let|\midverT #1}\,\Big\rbrace}}
}
\def\midvert{\egroup\mid\bgroup}
\def\midverT{\egroup\,\Big|\,\bgroup}

\def\Set[#1]#2|#3|{\Big\{\ #2\ \Big| \
           \vcenter{\hsize #1mm\centering #3}\Big\}}

\def\Hom{{\rm Hom}}

\def\mfg{{\mathfrak g}}

\def\Set{{\rm Set}}

\def\t{\mathfrak t}%
\def\Hom{\text{Hom}}%
\def\textsf#1{{\textit{#1}}}%

\begin{document}
\title{Blocks of the Brauer category over the complex field}
\author{ Mengmeng Gao, Hebing Rui, Linliang Song}
\address{M.G.  School of Mathematical Science, Tongji University,  Shanghai, 200092, China}\email{1810414@tongji.edu.cn}
\address{H.R.  School of Mathematical Science, Tongji University,  Shanghai, 200092, China}\email{hbrui@tongji.edu.cn}
\address{L.S.  School of Mathematical Science, Tongji University,  Shanghai, 200092, China}\email{llsong@tongji.edu.cn}
\thanks{H. Rui is supported  partially by NSFC (grant No.  11971351).  L. Song is supported  partially by NSFC (grant No.  12071346). }
\sloppy
\begin{abstract} Let $\B(\delta)$ be the Brauer category over the complex field $\mathbb C$ with the parameter $\delta$.    In non-semisimple case, $\delta$ is an integer, and  each weight space of $(\frac{\delta}2-1)$th semi-infinite wedge space corresponds to either a single  block or a union of two different  blocks of  $\B(\delta)$-lfdmod, the category of the locally finite dimensional representations of  $\B(\delta)$. Furthermore, each block contains an infinite number of irreducible representations of   $\B(\delta)$ and all blocks of  $\B(\delta)$-lfdmod can be obtained in this way.\\

  \end{abstract}\maketitle
\section{Introduction}
The idea of categorification establishes a nice connection between representations of various associative algebras and Lie theory. A famous example of categorification is the Ariki categorification theorem  which concerns the representations of  cyclotomic Hecke algebras over the complex number field $\mathbb C$~\cite{Ari}. 
Let $\Bbbk$ be an algebraically closed  field.  For any $d\in\mathbb N$ and $q\in \Bbbk^\times$, let $ H^\omega_d(\Bbbk ,q)$ be the cyclotomic Hecke algebra  over    $\Bbbk$ with respect to the dominant integral weight $\omega$ for the Kac-Moody algebra $\mathfrak g$,
where
$$ \mathfrak g=\left\{
                 \begin{array}{ll}
                   \widehat{\mathfrak {sl}}_e, & \hbox{if $q$ is a primitive $e$th root of unity,} \\
                   \mathfrak{sl}_\infty, & \hbox{if $q$ is not a root of unity.}
                 \end{array}
               \right.
$$
Let $H^\omega_d(\Bbbk ,q)$-mod be the category of finite dimensional representations of $H^\omega_d(\Bbbk ,q)$.
The Ariki categorification theorem  says that there is an isomorphism of  $\mathfrak g$-modules
\begin{equation}\label{cateofcyc}
\bigoplus_{d\geq 0}\mathbb C\otimes _{\mathbb Z}K_0(H^\omega_d(\mathbb  C,q)\text{-mod} )\cong V(\omega),
\end{equation}
 where $V(\omega)$ is the integrable highest weight $\mathfrak g$-module with highest weight $\omega$~\cite{Ari}. Similar isomorphism holds for the  cyclotomic Hecke algebra  over  $\Bbbk$~\cite[Theorem~4.18]{BK1}. This categorification  connects various important information on  $H^\omega_d(\Bbbk,q)$-mod to important invariants of the $\mathfrak g$-module $V(\omega)$. For details, see the survey article~\cite{Kle}. In particular,
\begin{itemize}
\item[(H1)]the actions of the Chevalley generators $e_i$, $f_i$ of $\mathfrak g$  correspond to the functors of $i$-induction and $i$-restriction on $ H^\omega_d(\Bbbk,q)$-mod;
\item[(H2)]when $\Bbbk=\mathbb C$, the classes of indecomposable projective $H^\omega_d(\mathbb C,q)$-modules correspond to the  canonical basis of $V(\omega)$;
\item[(H3)] each single  block of the cyclotomic Hecke algebra $ H^\omega_d(\Bbbk ,q)$ corresponds to a $\mathfrak g$-weight space of    $V(\omega)$.
\end{itemize}
There are  some other connections. For example,   the classes of irreducible  $H^\omega_d(\mathbb C,q)$-modules corresponds to the  dual canonical basis of $V(\omega)$, and the decomposition matrices of cyclotomic Hecke algebras $H^\omega_d(\mathbb  C,q)$ at the root of unity can be computed explicitly ~\cite{Ari}.  When $\omega=\omega_0$, the fundamental weight indexed by 0, such a result was  conjectured by   Lascoux, Leclerc and Thibon~\cite{LLT}.


This  paper is intended to supplement the paper \cite{RS2}  on the analog result of the Ariki  categorification theorem for the Brauer category. This is  a part of our project for studying categorifications related to the finite dimensional algebras arising from Schur-Weyl dualities in types $B,C$ and $D$.

Suppose $p$ is the characteristic of $\Bbbk$ and $p\neq 2$.
 Let $\B(\delta)$ be the Brauer category over $\Bbbk$ with the parameter $\delta\in \Bbbk$~\cite{LZ}.  Define complex  Lie algebra
  $$\mathfrak {sl}_\Bbbk= \left\{
                 \begin{array}{ll}
                   \widehat{\mathfrak {sl}}_p, & \hbox{if $p>0$,} \\
                   \mathfrak{sl}_\infty, & \hbox{if $p=0$.}
                 \end{array}
               \right.
$$
such that the Chevalley generators of $\mathfrak {sl}_\Bbbk$ are indexed by
\begin{equation}\label{ii} I:=\frac{\delta-1}{2}+\mathbb Z1_\Bbbk.\end{equation} Let $V(\omega_{\frac{\delta-1}{2}})$ be the
integrable highest weight $\mathfrak {sl}_\Bbbk$-module with highest weight $\omega_{\frac{\delta-1}{2}}$, the fundamental weight indexed by $\frac{\delta-1}{2}$. If $\delta\in \mathbb Z1_\Bbbk$, we  define  $\mathfrak g_{\delta}$ to be the subalgebra of $\mathfrak {sl}_\Bbbk$ generated by all $e_i+f_{-i}$, $i\in I$. In this case, $(\mathfrak{sl}_\Bbbk$, $\mathfrak g_\delta)$ forms a classical symmetric pair.  If $\delta\notin \mathbb Z1_\Bbbk$, let $\mathfrak g_\delta= \mathfrak {sl}_\Bbbk$. In  any case, $V(\omega_{\frac{\delta-1}{2}})$ is a $\mathfrak g_{\delta}$-module.
It is proved in  \cite[Theorem~4.13]{RS2} that the category $\B(\delta)$-lfdmod of  the locally finite dimensional left representations  of $\B(\delta)$  is an upper finite fully  stratified category in the sense of Brundan and Stroppel\cite[Definition~3.42]{BS}. Furthermore, as $\mathfrak g_\delta$-modules,
\begin{equation}\label{jujshxshx}
\mathbb C\otimes K_0(\text{$\B(\delta)$-lfdmod}^{\Delta})\cong V(\omega_{\frac{\delta-1}{2}}),
\end{equation}
where $\B(\delta)$-lfdmod$^{\Delta}$ is the subcategory of $\B(\delta)$-lfdmod in which each  object  admits a finite filtration with subquotients isomorphic to standard objects. If $\delta\notin \mathbb Z1_\Bbbk$, then $\B(\delta)$-lfdmod is Morita equivalent to
$\bigoplus_{d\geq0}\Bbbk S_d\text{-mod}$~\cite[Corollary~3.15]{RS2}, where $S_d$ is the symmetric group on $d$ letters.
In this case,  \eqref{jujshxshx} is  equivalent to \eqref{cateofcyc} for symmetric groups.
When $\delta\in \mathbb Z1_\Bbbk$, we have  analog correspondences of (H1),(H2) as follows.
\begin{itemize}
\item [(B1)] the actions of the   generators $e_i+f_{-i}$ of $\mathfrak g_\delta$  correspond to the functors of $i$-induction   on  $\B(\delta)$-lfdmod~\cite[Theorem~4.3]{RS2};
\item [(B2)] If $\Bbbk=\mathbb C$, the  classes of indecomposable projective modules in $\B(\delta)$-lfdmod correspond to quasi-canonical basis  of $V(\omega_{\frac{\delta-1}{2}})$~\cite[Theorem~5.6]{RS2}.
\end{itemize}

We did not obtain  an analog result of (H3) for  $\B(\delta)$ in \cite{RS2}. Using the $\mathfrak g_\delta$-weight space  of    $V(\omega_{\frac{\delta-1}{2}})$  under the categorification in \eqref{jujshxshx}, we give  a partial result on blocks of $\B(\delta)$-lfdmod  in \cite[Lemma~4.9(2)]{RS2}. The  aim  of the paper is to give the following results on
  blocks of   $\B(\delta)$-lfdmod over $\mathbb C$,   an analog result of (H3).  Explicit description can be found in  Theorem~\ref{main}.
\medskip
\begin{enumerate}
\item [(B3)]\begin{itemize}\item If  $\delta$ is odd, then each  $\mathfrak g_\delta$-weight space of  $V(\omega_{\frac{\delta-1}{2}})$ corresponds to a single block of $\B(\delta)$-lfdmod;
\item If  $\delta$ is even, then any $\mathfrak g_\delta$-weight space of  $V(\omega_{\frac{\delta-1}{2}})$   corresponds to either  a single block or a  union of two different  blocks of $\B(\delta)$-lfdmod.
\end{itemize}\end{enumerate}
In any case, each block contains an infinite number of irreducible representations of $\B(\delta)$.

The content of this note is organized as follows. In section~2, we give some results on $\mathfrak g_\delta$-weight spaces of certain  semi-infinite wedge spaces in \cite{BW}. In section~3, we prove results on the blocks of $\B(\delta)$-lfdmod over $\mathbb C$.
Finally, we give a relationship between the $\mathfrak g_\delta$-weight spaces and the ``central blocks'' defined by a subalgebra of the center of $\B(\delta)$. This is motivated by Brundan-Vargas's work  on partition categories~\cite{BV}.
\vskip0.5cm
\textbf{Acknowledgement}: We thanks the referee for his/her detailed comments on this paper.

\section{ Semi-infinite wedge spaces  }
From here to the end of this paper, the ground field $\Bbbk$ is the complex number field $\mathbb C$.  All categories, algebras etc will be defined over  $\mathbb C$ and  all modules are left modules.

\subsection{Lie algebra $\mathfrak{sl}_\infty$}\label{lie}  Fix an integer $\delta$ and define
 \begin{equation}\label{io} \mathbb I=\frac12+I, \end{equation}  where $ I$ is given in \eqref{ii}. So,
 $$\mathbb I =\begin{cases} \mathbb Z &\text{if $\delta$ is even,}\\
 \frac12+\mathbb Z &\text{if $\delta$ is odd.}\end{cases}$$  Consider the  Cartan matrix $(a_{i,j})_{i,j\in I}$
such that $$
a_{i,j}=\begin{cases}
            2 & \text{if $i=j$, } \\
            -1 & \text{if $|i-j|= 1$,} \\
            0 & \text{otherwise.}
          \end{cases}
$$
 Let $\mathfrak{sl}_\infty$ be  the   Kac-Moody algebra associated to $(a_{i,j})_{i,j\in I}$.
  Let $\mathfrak h$ be its  Cartan subalgebra. Define $h_i=[e_i,f_i]$, where $\{e_i, f_i\mid i\in I\}$ is  the set of usual Chevalley generators of  $\mathfrak {sl}_{\infty}$.  Let $\mathfrak h^*$ be the linear dual of $\mathfrak h$. Throughout, we fix the  notations as follows:
\begin{itemize}
\item $P= \{\lambda\in\mathfrak h^*\mid \langle h_i,\lambda \rangle\in \mathbb Z,  \forall i\in I\}=\sum_{i\in  I}\mathbb Z\omega_i$, where $\omega_i$'s are fundamental weights,
 \item $\Pi=\{\alpha_i\mid i\in I\}$,  where  $\alpha_i=\varepsilon_{i-\frac{1}{2}}-\varepsilon_{i+\frac{1}{2}}$ and  $\varepsilon_{i-\frac{1}{2}}=\omega_{i}-\omega_{i-1}$, $\forall i\in I$,
 \item $Q=\sum_{i\in I}\mathbb Z \alpha_i$.
\end{itemize}
\subsection{Symmetric pairs}\label{SP} 
 Consider the graph automorphism $\theta$ of Dynkin diagram associated to $I$  in Figures 1--2:
 \begin{figure}[ht!]
   \label{figure:ji3}
\begin{tikzpicture}
 \draw[dotted] (-1,0) --(0,0);\draw[dotted] (6,0) --(7,0);
 \draw[dotted]  (0.5,0) node[below]  {$ -m+\frac{1}{2}$} -- (2.5,0) node[below]  {$-\frac{1}{2}$} ;
 \draw (2.5,0)
 -- (3.5,0) node[below]  {$\frac12$};
 \draw[dotted] (3.5,0) -- (5.5,0) node[below] {$m-\frac12$} ;
\draw (0.5,0) node (-m) {$\bullet$};
 \draw (2.5,0) node (-1) {$\bullet$};
\draw (3.5,0) node (1) {$\bullet$};
\draw (5.5,0) node (m) {$\bullet$};
\draw[<->] (-m.north east) .. controls (3,1) .. node[above] {$\theta$} (m.north west) ;
\draw[<->] (-1.north) .. controls (3,0.5) ..  (1.north) ;
\end{tikzpicture}
\caption{Dynkin diagram of type $A_{\infty}$ with involution   $\theta$ when $\delta$ is even.}
\end{figure}

\begin{figure}[ht!]

\begin{tikzpicture}
 \draw[dotted] (-1,0) --(0,0);\draw[dotted] (6,0) --(7,0);
 \draw[dotted]  (0,0) node[below] {$-m$} -- (2,0) node[below] {$-1$} ;
 \draw (2,0) -- (3,0) node[below]  {$0$}
 -- (4,0) node[below] {$1$};
 \draw[dotted] (4,0) -- (6,0) node[below] {$m$} ;
\draw (0,0) node (-m) {$\bullet$};
 \draw (2,0) node (-1) {$\bullet$};
\draw (3,0) node (0) {$\bullet$};
\draw (4,0) node (1) {$\bullet$};
\draw (6,0) node (m) {$\bullet$};
\draw[<->] (-m.north east) .. controls (3,1.5) .. node[above] {$\theta$} (m.north west) ;
\draw[<->] (-1.north) .. controls (3,1) ..  (1.north) ;
\draw[<->] (0) edge[<->, loop above] (0);
\end{tikzpicture}
\caption{Dynkin diagram of type $A_{\infty}$ with involution   $\theta$ when $\delta$ is odd.}
   \label{figure:ji4}
\end{figure}
\noindent Then   $\theta$ sends the $i$th vertex to $-i$th vertex for  any $i\in I$.
 It gives  rise to the involution on   $P$ satisfying
$\theta(\varepsilon_i )=-\varepsilon_{-i}$, for all $ i\in \mathbb I$.
 Thanks to  \cite[(4.19)]{KW}, it
 induces  an automorphism of $\mathfrak {sl}_{\infty} $, denoted by $\theta$  such that
$\theta(e_i)=e_{-i}$, $\theta(f_i)=f_{-i}$, $\theta(h_i)=h_{-i}$ and
 $\langle\theta(h),\alpha_{-i}\rangle=\langle h, \alpha_{i} \rangle$.
Moreover, $\theta$ induces a linear map on $\mathfrak h^*$ such that $\theta(\alpha_i)=\alpha_{-i}$.
Consider the automorphism $\phi$ of $\mathfrak {sl}_\infty$ such that $\phi$ switches $e_i$ and $f_i$ for all admissible $i$ and sends $h$ to $-h$ for any $h\in \mathfrak h$. Then $\phi \circ \theta$  is an automorphism of $\mathfrak{sl}_\infty$ such that
$\phi \circ \theta (e_i)=f_{-i}$, $\phi \circ \theta (f_i)=e_{-i}$ and $\phi \circ \theta(h_i)=-h_{-i}$, $\forall  i\in I$.
Following \cite{BW,ES}, let   \begin{equation}\label{gdelta}
\mathfrak g_\delta=\{g\in\mathfrak {sl}_\infty\mid \phi \circ \theta (g)=g\}.\end{equation} It is  the Lie subalgebra of $\mathfrak{sl}_\infty$ generated by   $\{e_i+f_{-i}\mid i\in I\}$,
 and $(\mathfrak{sl}_\infty$, $\mathfrak g_\delta)$ forms a classical symmetric pair.

\subsection{Semi-infinite wedge space} Let $\mathbb V$ be the  natural representation  of $\mathfrak {sl}_\infty$ over $\mathbb C$.
It is the  $\mathbb C$-space with basis
$\{v_i\mid i\in\mathbb I \}$ such that
\begin{equation}\label{eexheu}
e_i v_a=\delta_{i+\frac{1}{2}, a}v_{a-1}, \quad f_i v_a=\delta_{i-\frac{1}{2},a}v_{a+1}, \forall (i, a)\in I\times \mathbb I.
\end{equation}
Consider  the restricted dual $\mathbb W$ of $\mathbb V$. As a vector space, it  has  dual  basis $\{w_a\mid a\in\mathbb I\}$ such that
 \begin{equation} e_i w_a=\delta_{i-\frac{1}{2}, a}w_{a+1}, \quad f_i w_a=\delta_{i+\frac{1}{2},a}w_{a-1}, \forall (i, a ) \in I\times \mathbb I.\end{equation}
  As a vector of the $\mathfrak {sl}_\infty$-module  $\mathbb V$  (resp., $\mathbb W$),  $v_i$ (resp., $w_i$) has the weight  $\epsilon_i$ (resp., $-\epsilon_i$).
  Restricting $\mathbb V$ (resp.,  $\mathbb W$) to $\mathfrak g_\delta$ yields  a $\mathfrak g_\delta$-module. Bao and Wang~\cite{BW} considered $\mathfrak{sl}_\infty$-modules (and hence $\mfg_\delta$-modules)
$$\textstyle\bigwedge^{\infty}\mathbb V=\bigoplus_{d\in \mathbb I}\bigwedge_d^\infty \mathbb V \text{ and }   \textstyle\bigwedge^{\infty}\mathbb W=\bigoplus_{d\in \mathbb I}\bigwedge_d^\infty \mathbb W.$$ The submodule $\bigwedge_d^{\infty}\mathbb V$
 (resp., $\bigwedge_d^{\infty}\mathbb W$)  is called   the $d$th sector of  semi-infinite wedge space $\bigwedge^{\infty}\mathbb V$ (resp.,$\bigwedge^{\infty}\mathbb W$). Let $\mathbb I^\infty=\{  \mathbf (i_1,i_2,\ldots)\mid   i_k\in \mathbb I,  \forall \text{ positive integers  $k$}\}$. Later on, $(i_1,i_2,\ldots)$ will be denoted by $\mathbf i$.
Define  \begin{equation}\label{id321}
\mathbb I_d^-=\{\mathbf i\in\mathbb I^\infty\mid i_1<i_2<\ldots, i_k=d+k  \text{ for $k\gg0$}\}.
\end{equation} For each $\mathbf i\in \mathbb I_d^-$, define \begin{equation}\label{wbi}w_{\mathbf i}=w_{i_1}\wedge w_{i_2}\wedge\ldots.\end{equation} Then
 $\bigwedge_d^{\infty}\mathbb W$  has basis   $\{w_{\mathbf i}\mid \mathbf i\in \mathbb I_d^-\}$.

A partition $\lambda$ of a nonnegative integer $n$ is a weakly decreasing sequence $(\lambda_1, \lambda_2, \ldots)$ of non-negative integers  such that $\sum_j \lambda_j=n$.  Define \begin{equation}\label{lambda} \Lambda=\bigcup_{n=0}^\infty \Lambda_n,\end{equation} where  $\Lambda_n$ is  the set of all partitions of $n$.
 For any $\lambda\in \Lambda$, define
 \begin{equation}\label{ilambda}\mathbf i_{d, \lambda}=(d-\lambda_1+1, d-\lambda_2+2, d-\lambda_3+3, \ldots )\in \mathbb I^\infty.\end{equation}
 Obviously, $\mathbf i_{d, \lambda}\in \mathbb I_d^-$ and the map from $\Lambda$ and $\mathbb I_d^-$
  sending $\lambda$ to $\mathbf i_{d,\lambda}$ is bijective.

 \begin{Defn}\label{phi}   For any $\mathbf i\in \mathbb I_d^-
$, define $\text{wt}_d(\mathbf i):=\text{wt}_d(w_{\mathbf i})=-\sum_{j=1}^\infty\varepsilon_{i_j}$.\end{Defn}
  In fact,
$\text{wt}_d(w_{\mathbf i})$ is  the weight of $w_{\mathbf i}$ when it is considered as a vector of $\mathfrak {sl}_\infty$-module  $\bigwedge_d^{\infty}\mathbb W$.

\subsection{Weyl group of type $D_\infty$}
For $n\geq2$, consider the free abelian group $\bigoplus_{i=1}^n \mathbb Z\delta_i$ on which there is a symmetric bilinear form $(\ ,\ )$  such that
$$(\delta_i,\delta_j)=\delta_{i,j},  \ \ 1\le i, j\le n.$$
Then the root system $R=\{\pm(\delta_{i}-\delta_j), \pm(\delta_{i}+\delta_j) \mid 1\le i<  j\le n\}$. Let $E=R\otimes_{\mathbb Z}  \mathbb R$. For each $\alpha\in R$, the corresponding reflection $s_\alpha$ is defined as
\begin{equation}\label{sr} s_\alpha (v)=v-(v, \alpha)\alpha, \forall v\in E.\end{equation}
Fix the simple root system  $\Pi_n=\{\beta_0,\beta_1, \beta_2, \ldots, \beta_{n-1}\}$, where
  $\beta_i=\delta_i-\delta_{i+1}$,  $ 1\le i\le n-1$, and $\beta_0=-\delta_1-\delta_2$.
The corresponding  Dynkin diagram of type $D_n$ together with the set of simple roots $\Pi_n$ are listed as follows:
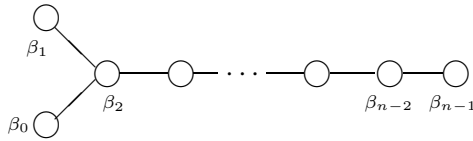
\begin{center}
\hskip 3cm \setlength{\unitlength}{0.16in}
\begin{picture}(24,3.5)
\put(8,2){\makebox(0,0)[c]{$\bigcirc$}}
\put(10.4,2){\makebox(0,0)[c]{$\bigcirc$}}
\put(14.85,2){\makebox(0,0)[c]{$\bigcirc$}}
\put(17.25,2){\makebox(0,0)[c]{$\bigcirc$}}
\put(19.4,2){\makebox(0,0)[c]{$\bigcirc$}}
\put(6,3.8){\makebox(0,0)[c]{$\bigcirc$}}
\put(6,.3){\makebox(0,0)[c]{$\bigcirc$}}
\put(8.4,2){\line(1,0){1.55}} \put(10.82,2){\line(1,0){0.8}}
\put(13.2,2){\line(1,0){1.2}} \put(15.28,2){\line(1,0){1.45}}
\put(17.7,2){\line(1,0){1.25}}
\put(7.6,2.2){\line(-1,1){1.3}}
\put(7.6,1.8){\line(-1,-1){1.3}}
\put(12.5,1.95){\makebox(0,0)[c]{$\cdots$}}
\put(5.1,0.3){\makebox(0,0)[c]{\tiny$\beta_{0}$}}
\put(5.7,2.8){\makebox(0,0)[c]{\tiny$\beta_{1}$}}
\put(8.2,1){\makebox(0,0)[c]{\tiny$\beta_{2}$}}
\put(17.2,1){\makebox(0,0)[c]{\tiny$\beta_{n-2}$}}
\put(19.3,1){\makebox(0,0)[c]{\tiny$\beta_{n-1}$}}
\end{picture}.
\end{center}
The reflection group $W_n$ generated by $\{s_\beta\mid \beta\in R\}$ is the Weyl group  of type $D_n$. It is the Coxeter group with distinguished generators $s_i:=s_{\beta_i}$, $i=0,1,\ldots, n-1$. In this setting, $ W_n$ is the standard parabolic subgroup of $W_{n+1}$. The required monomorphism sends $s_i\in W_n$ to $s_i\in W_{n+1}$ for all $i, 0\le i\le n-1$.
Let $W_\infty$ be the Weyl group of type $D_\infty$. Any $W_n$ can be considered as a parabolic subgroup of $W_\infty$.

Consider the orthogonal Lie algebra $\mathfrak {so}_{2n}$ for $n\ge 2$. There is a triangular decomposition   $\mathfrak {so}_{2n}=\mathfrak n_n^{-}\oplus \mathfrak t_n\oplus \mathfrak n_n^+$ corresponding to  $\Pi_n$.
Let $\mathfrak t_n^*$ be the linear dual  of $\mathfrak t_n$ with  dual basis $\{\delta_i\mid 1\le i\le n\}$.
Each element $\lambda\in \mathfrak t_n^*$, called a weight, is of form   $\sum_{i=1}^n \lambda_i\delta_i$. $\lambda$  is integral if it is either  $\mathbb Z$-span or $(\mathbb Z+\frac{1}{2})$-span of the $\delta_i$'s. Similarly, we have  $\mathfrak {so}_{2\infty}$ and  $\mathfrak t_\infty^*$, etc.
Each $\lambda\in \mathfrak t_n^*$ can be considered as the element in $\mathfrak t_\infty^*$ by setting $\lambda_j=0$ for all $j\ge n+1$. The Weyl group $W_\infty$ can act on $\mathbb I^\infty$ since
any $\mathbf i\in \mathbb I^\infty$ can be considered as the integral weight $\sum_{j=1}^\infty i_j\delta_j\in \mathfrak t_\infty^*$. Write   \begin{equation}\label{aaa} w( \mathbf i)=\mathbf j, \text{ for any $\mathbf i, \mathbf j\in \mathbb I^\infty$ if $w(\sum_{k=1}^\infty i_k\delta_k)=\sum_{k=1}^\infty j_k\delta_k$.}\end{equation} The following result follows from \eqref{sr}, immediately.

\begin{Lemma} \label{key000} Suppose $\mathbf i\in \mathbb I^\infty$. For all positive integers  $j, k$ with $j<k$, we have  \begin{itemize}\item[(1)]  $s_{\delta_j-\delta_k}(\mathbf i)$ is obtained from $\mathbf i$  by switching $i_j$ and $i_k$ and fixing other terms,
\item[(2)]  $s_{\delta_j+\delta_k}(\mathbf i)$ is obtained from  $\mathbf i$ by replacing  $i_j, i_k$ by $-i_k, -i_j$ and  fixing other terms.\end{itemize}\end{Lemma}

\subsection{$\mfg_\delta$-weight spaces of  $\bigwedge_d^\infty \mathbb W$ }\label{hhgweight}  Recall  $P, Q$ in subsection~\ref{lie} and $\theta$ in subsection~\ref{SP}.
Let  $P_\theta=P/Q^\theta$ where $Q^\theta=\{ \theta(\mu)+\mu\mid \mu\in Q\}$. For any $\mathbf i\in \mathbb I_d^-$, $\bar{\text{wt}_d(\mathbf i)}$ is known as a $\mathfrak g_\delta$-weight of $\mathfrak g_\delta$-module  $\bigwedge_d^\infty \mathbb W$~\cite{BW}, where $\bar{\text{wt}_d(\mathbf i)}$ is the image of $\text{wt}_d(\mathbf i)$ in $P_\theta$. So, we have a weight function $\bar{\text{wt}_d}: \mathbb I_d^-\rightarrow P_\theta$ such that $\bar{\text{wt}_d} (\mathbf i)=\bar{\text{wt}_d(\mathbf i)}$.
\begin{Prop}\label{o1} Suppose $ \mathbb I=\frac{1}{2}+\mathbb Z$. For any $\mathbf i, \mathbf j\in \mathbb I_{d}^-$, we have   $ {\bar{\text{wt}_d}(\mathbf i)}= {\bar{\text{wt}_d}(\mathbf j) }$ if and only if  $w (\mathbf i)=\mathbf j$ for some $w\in W_\infty$. \end{Prop}
\begin{proof} Assume  $\mathbf i=(i_1, i_2, \ldots)\in \mathbb I_d^-$.  Since we are assuming $d\in \mathbb I= \frac12+\mathbb Z$,  $i_j\neq 0$ for all $j\ge 1$. Suppose $w(\mathbf i)=\mathbf j$ for some   $w\in W_\infty$. Since $w$  can be expressed as a product of certain $s_{\delta_k-\delta_l}$'s  and    $s_{\delta_k+\delta_l}$'s,   by Lemma~\ref{key000},
there are some $b_{i,j}\in \mathbb{Z}$ such that
$$\text{wt}_d(\mathbf i)-\text{wt}_d(
\mathbf j)=\sum_{i, j\in\mathbb{ I},i\neq -j}b_{i,j}(\varepsilon_i+\varepsilon_{j}-\varepsilon_{-i}-\varepsilon_{-j}).$$
 We have $\varepsilon_i+\varepsilon_{j}-\varepsilon_{-i}-\varepsilon_{-j}\in Q^\theta$ since

\begin{equation}\label{wtt2} \varepsilon_i+\varepsilon_{j}-\varepsilon_{-i}-\varepsilon_{-j}=\begin{cases} \sum_{k=i}^{-j-1}( \alpha_{k+\frac{1}{2}}+ \alpha_{-(k+\frac12)}), & \text{if $i<-j$,}\\ \sum_{k=-j+1}^i (\alpha_{k-\frac12}+\alpha_{-(k-\frac{1}{2})})
, & \text{if $i>-j$.}\\
\end{cases} \end{equation} In any case,  $ {\bar{\text{wt}_d}(\mathbf i)}= {\bar{\text{wt}_d}(\mathbf j) }$.
Conversely, if $ {\bar{\text{wt}_d}(\mathbf i)}= {\bar{\text{wt}_d}(\mathbf j) }$,  then  $$\text{wt}_d(\mathbf i)-\text{wt}_d(\mathbf j)=\sum_{j\in \mathbb{I}}a_j\varepsilon_j=\sum_{i\in I,i\geq0}b_i(\alpha_i+\alpha_{-i})$$ for some $a_j\in\{1,-1,0\}$ and $b_i\in \mathbb Z$.
Since
\begin{equation}\label{key12345}\alpha_i+\alpha_{-i}=\varepsilon_{i-\frac{1}{2}}-\varepsilon_{i+\frac{1}{2}}+
\varepsilon_{-i-\frac{1}{2}}-\varepsilon_{-i+\frac{1}{2}}\end{equation}
and
 \begin{equation}\label{minues}
 \alpha_i+\alpha_{-i}+\alpha_{i+1}+\alpha_{-i-1}=\varepsilon_{i-\frac{1}{2}}-\varepsilon_{-i+\frac{1}{2}}-\varepsilon_{i+\frac{3}{2}}+\varepsilon_{-i-\frac{3}{2}}, \text{for $i\ge 0$},
 \end{equation}
 we have $a_j=-a_{-j}$ and $\sharp\{j| a_j\neq 0\}$ is even. In other words, $\mathbf j$ is obtained from $\mathbf i$ by  permuting its entries and  changing even number of signs.  Thanks to Lemma~\ref{key000},  $\mathbf j=w(\mathbf i)$ for some  $w\in W_\infty$.
\end{proof}

\begin{Prop}\label{o11} Suppose $\mathbb I= \mathbb Z$ and  $\mathbf i, \mathbf j\in \mathbb I_{d}^-$ such that there is a unique $k$ with $i_k=0$. Then  $ {\bar{\text{wt}_d}(\mathbf i)}= {\bar{\text{wt}_d}(\mathbf j) }$    if and only if   $w (\mathbf i)=\mathbf j$ for some  $w\in W_\infty$. \end{Prop}
\begin{proof}
Suppose  $w(\mathbf i)=\mathbf j$ for some $w\in W_\infty$. Thanks to Lemma~\ref{key000},  there are some $b_{i}\in \mathbb{Z}$ such that
$$\text{wt}_d(\mathbf i)-\text{wt}_d(
\mathbf j)=\sum_{i>0}b_{i}(\varepsilon_i- \varepsilon_{-i}).$$
Since $i_k=0$, $\sharp\{i\mid    b_i\neq 0\}$ may be odd.
For any $i>0$,  we have

\begin{equation}\label{wtt21} \varepsilon_{-i}-\varepsilon_i = \sum_{j=\frac12}^{ i-\frac{1}{2}}( \alpha_j+\alpha_{-j}) \in Q^\theta.  \end{equation}
So,  $ {\bar{\text{wt}_d}(\mathbf i)}= {\bar{\text{wt}_d}(\mathbf j) }$.
Conversely,  if  $ {\bar{\text{wt}_d}(\mathbf i)}= {\bar{\text{wt}_d}(\mathbf j) }$, then  $$\text{wt}_d(\mathbf i)-\text{wt}_d(\mathbf j)=\sum_{j\in \mathbb{I}}a_j\varepsilon_j=\sum_{i\in I,i>0}b_i(\alpha_i+\alpha_{-i})$$ for some $a_j\in\{1,-1,0\}$ and $b_i\in \mathbb Z$.
Obviously $\alpha_{\frac12}+\alpha_{-\frac12}=\varepsilon_{-1}-\varepsilon_1$.
If $b_{\frac12}$ is even,  by \eqref{key12345}--\eqref{minues}, $a_j=-a_{-j}$ and $\sharp\{j| a_j\neq 0\}$ is even. So,  $\mathbf j$ is obtained from $\mathbf i$ by permuting its entries and  replacing the entries $k, l$ by $-l, -k$ for some pairs $k, l\neq0$. Thanks to Lemma~\ref{key000}, $\mathbf j=w(\mathbf i)$ for some  $w\in W_\infty$.
If $b_{\frac12}$ is odd,  by \eqref{key12345}--\eqref{minues}, $a_j=-a_{-j}$ and  $\sharp\{j| a_j\neq 0\}$ is odd. Since $i_k=0$, $\mathbf j$ is obtained from $\mathbf i$ by permuting its entries and  replacing the entries $ l, m$ by $-m, -l$ for some pairs $l, m$. In this case, one of $l, m$ may be $0$. Thanks to Lemma~\ref{key000} again, we have   $\mathbf j=w(\mathbf i)$ for some  $w\in W_\infty$.
\end{proof}

When $\mathbb I= \mathbb Z$, we consider the subgroup $Q^{\theta, +}$ of $Q^\theta$ such that  \begin{equation}\label{q1} Q^{\theta, +}=2\mathbb{Z}(\alpha_{\frac{1}{2}}+\alpha_{-\frac{1}{2}})+ \sum_{i\in I\setminus\{ \frac{1}{2}, -\frac12\}}\mathbb{Z}(\alpha_i+\alpha_{-i}).\end{equation} Then $Q^\theta=Q^{\theta, +}\cup Q^{\theta, -}$ where $Q^{\theta, -}= (\alpha_{\frac{1}{2}}+\alpha_{-\frac{1}{2}})+Q^{\theta, +}$.
\begin{Prop}\label{o2} Suppose $\mathbb I=\mathbb Z$ and   $\mathbf i, \mathbf j\in \mathbb I_{d}^-$  such that   $i_k\neq 0$ for all $k\ge 1$. We have   $ {\text{wt}_d(\mathbf i)}\equiv \text{wt}_d(\mathbf j) \pmod
{Q^{\theta, +}}$ if and only if  $w (\mathbf i)=\mathbf j$ for some $w\in W_\infty$.
\end{Prop}
\begin{proof}
 If $\mathbf j=w (\mathbf i)$ for some $w\in W_\infty$, then there are some integers $b_{i,j}$'s such that

 $$\text{wt}_d(\mathbf i)-\text{wt}_d(\mathbf j)=\sum_{i,j\in\mathbb{ I}\setminus 0,i\neq -j}b_{i,j}(\varepsilon_i+\varepsilon_{j}-\varepsilon_{-i}-\varepsilon_{-j}).$$
 Thanks to  \eqref{wtt21},
 $\text{wt}_d(\mathbf i)\equiv \text{wt}_d(\mathbf j)\pmod{ Q^{\theta, +}}$. Conversely, if $\text{wt}_d(\mathbf i)\equiv \text{wt}_d(\mathbf j)\pmod{ Q^{\theta, +}}$, then
$$\text{wt}_d(\mathbf i)-\text{wt}_d(\mathbf j)=\sum_{j\in\mathbb{I}}a_j\varepsilon_j=2b_{\frac12}(\alpha_{\frac{1}{2}}+\alpha_{-\frac{1}{2}})+ \sum_{i\in I, i> \frac{1}{2}}b_i(\alpha_i+\alpha_{-i})$$ for some $a_j\in\{1,-1,0\}$ and $b_i\in \mathbb Z$.
Thanks to  \eqref{key12345}--\eqref{minues},
 $a_j=-a_{-j}$ and $\sharp\{j| a_j\neq 0\}$ is even.
By Lemma~\ref{key000},  $\mathbf j=w(\mathbf i)$ for some  $w\in W_\infty$.
\end{proof}

\begin{Theorem}\label{weyl}For any $\mathbf i\in \mathbb I_d^-$, define  $\mathbf j=(-i_k, i_1, i_2, \ldots, i_{k-1}, i_{k+1}, \ldots)$ for $k$ sufficiently large and  $\gamma=\bar{\text{wt}_d}(\mathbf i)\in P_\theta$. Then $\mathbf j\in \mathbb I_d^{-}$ and $\bar{\text{wt}_d}(\mathbf j)=\gamma$. Moreover,

\begin{enumerate}
\item  $\bar{\text{wt}_d}^{-1}(\gamma)=W_\infty\mathbf i\cap \mathbb I_d^-$ if either  $d\in \frac12+\mathbb Z$ or $d\in \mathbb Z$ and  $i_l=0$ for some $l$,
    \item  $\bar{\text{wt}_d}^{-1}(\gamma) =(W_\infty \mathbf i\cap \mathbb I_d^-) \overset{.}\cup (W_\infty \mathbf j\cap \mathbb I_d^-)$,
 if $d\in \mathbb Z$ and  $i_l\neq 0$ for all  $l$.
 \end{enumerate}
 Furthermore, $W_\infty \mathbf i\cap \mathbb I_d^-$ contains  infinite numbers of elements in  $\mathbb I_d^-$ in any case.
\end{Theorem}

\begin{proof}The first statement follows from \eqref{wtt21}. The remaining   statements  follow from Propositions~\ref{o1}--\ref{o2}, immediately. \end{proof}
\section{Blocks of  Brauer category $\B(\delta)$ over $\mathbb C$ }
\subsection{Brauer categories and affine Brauer categories} The affine Brauer category $\AB$ has been introduced in \cite{RS3}. It is the strict $\Bbbk$-linear  monoidal category
  generated by a single object
 \begin{tikzpicture}[baseline = 10pt, scale=0.5, color=\clr]
                \draw[-,thick] (0,0.5)to[out=up,in=down](0,1.2);
    \end{tikzpicture}, where $\Bbbk$ is an arbitrary domain.  In this note, we focus on the complex number field $\mathbb C$.
 The   objects of  $\AB$ are  $\ob m$'s, where $\ob m$ represents
 $
\begin{tikzpicture}[baseline = 10pt, scale=0.5, color=\clr]
                \draw[-,thick] (0,0.5)to[out=up,in=down](0,1.2);
    \end{tikzpicture}^{\otimes m} $, $m\in \mathbb N$.
     So, the unit object in $\AB$  is  $\ob 0$, while the object
\begin{tikzpicture}[baseline = 10pt, scale=0.5, color=\clr] \draw[-,thick] (0,0.5)to[out=up,in=down](0,1.2);
    \end{tikzpicture}
    is  $\ob 1$. There are four generating morphisms    $\lcup:\ob0\rightarrow \ob2$, $ \lcap:\ob2\rightarrow\ob0$, $ \swap: \ob2\rightarrow\ob2$   and  $\xdot:\ob1\rightarrow\ob1$. For any $m>0$, the identity morphism $1_{\ob m}: \ob m\rightarrow \ob m$  is drawn as the object itself. For example,  $\begin{tikzpicture}[baseline = 10pt, scale=0.5, color=\clr]
                \draw[-,thick] (0,0.5)to[out=up,in=down](0,1.2);\draw[-,thick] (0.5,0.5)to[out=up,in=down](0.5,1.2);
    \end{tikzpicture}$ represents  $\text{1}_{\ob 2}$.
  Using string calculus in a strict monoidal category,  we list  the defining relations  as follows~\cite[Definition~1.2]{RS3}:
     \begin{equation}\label{B1}
        \begin{tikzpicture}[baseline = 10pt, scale=0.5, color=\clr]
            \draw[-,thick] (0,0) to[out=up, in=down] (1,1);
            \draw[-,thick] (1,1) to[out=up, in=down] (0,2);
            \draw[-,thick] (1,0) to[out=up, in=down] (0,1);
            \draw[-,thick] (0,1) to[out=up, in=down] (1,2);
                    \end{tikzpicture}
        ~=~
        \begin{tikzpicture}[baseline = 10pt, scale=0.5, color=\clr]
            \draw[-,thick] (0,0) to (0,1);
            \draw[-,thick] (0,1) to (0,2);
            \draw[-,thick] (1,0) to (1,1);
            \draw[-,thick] (1,1) to (1,2);
        \end{tikzpicture}~
        ,\qquad
        \begin{tikzpicture}[baseline = 10pt, scale=0.5, color=\clr]
            \draw[-,thick] (0,0) to[out=up, in=down] (2,2);
            \draw[-,thick] (2,0) to[out=up, in=down] (0,2);
            \draw[-,thick] (1,0) to[out=up, in=down] (0,1) to[out=up, in=down] (1,2);
        \end{tikzpicture}
        ~=~
        \begin{tikzpicture}[baseline = 10pt, scale=0.5, color=\clr]
            \draw[-,thick] (0,0) to[out=up, in=down] (2,2);
            \draw[-,thick] (2,0) to[out=up, in=down] (0,2);
            \draw[-,thick] (1,0) to[out=up, in=down] (2,1) to[out=up, in=down] (1,2);
        \end{tikzpicture}~,
    \end{equation}
    \begin{equation}\label{B2}
        \begin{tikzpicture}[baseline = 10pt, scale=0.5, color=\clr]
            \draw[-,thick] (2,0) to[out=up, in=down] (2,1) to[out=up, in=right] (1.5,1.5) to[out=left,in=up] (1,1);
            \draw[-,thick] (1,1) to[out=down,in=right] (0.5,0.5) to[out=left,in=down] (0,1) to[out=up,in=down] (0,2);
        \end{tikzpicture}
        ~=~
        \begin{tikzpicture}[baseline = 10pt, scale=0.5, color=\clr]
            \draw[-,thick] (0,0) to (0,1);
            \draw[-,thick] (0,1) to (0,2);
        \end{tikzpicture}
        ~=~
        \begin{tikzpicture}[baseline = 10pt, scale=0.5, color=\clr]
            \draw[-,thick] (2,2) to[out=down, in=up] (2,1) to[out=down, in=right] (1.5,0.5) to[out=left,in=down] (1,1);
            \draw[-,thick] (1,1) to[out=up,in=right] (0.5,1.5) to[out=left,in=up] (0,1) to[out=down,in=up] (0,0);
        \end{tikzpicture}~,
    \end{equation}

\begin{equation}\label{B3}
    \begin{tikzpicture}[baseline = 5pt, scale=0.5, color=\clr]
        \draw[-,thick] (0,1) to[out=down,in=left] (0.5,0.35) to[out=right,in=down] (1,1);
    \end{tikzpicture}
    ~=~
    \begin{tikzpicture}[baseline = 5pt, scale=0.5, color=\clr]
        \draw[-,thick] (0,1) to[out=down,in=up] (1,0) to[out=down,in=right] (0.5,-0.5) to[out=left,in=down] (0,0) to[out=up,in=down] (1,1);
    \end{tikzpicture}
    ~,\qquad
    \begin{tikzpicture}[baseline = 5pt, scale=0.5, color=\clr]
        \draw[-,thick] (0,0) to[out=up,in=left] (0.5,0.65) to[out=right,in=up] (1,0);
    \end{tikzpicture}
    ~=~
    \begin{tikzpicture}[baseline = 5pt, scale=0.5, color=\clr]
        \draw[-,thick] (0,0) to[out=up,in=down] (1,1) to[out=up,in=right] (0.5,1.5) to[out=left,in=up] (0,1) to[out=down,in=up] (1,0);
    \end{tikzpicture}~,
    \end{equation}

    \begin{equation}\label{B4} \begin{tikzpicture}[baseline = -1mm, color=\clr]
	\draw[-,thick ] (0.28+0.56+0.56,0) to[out=90, in=0] (0+0.56+0.56,0.3);
	\draw[-,thick ] (0+0.56+0.56,0.3) to[out = 180, in = 90] (-0.28+0.56+0.56,0);
	
\draw[-,thick ] (0.28,-0.45) to[out=60,in=-90] (0.28+0.56,0);
	\draw[-,thick ](0.28+0.56,-0.45) to[out=120,in=-90] (0.28,0);
 \draw[-,thick ] (0.28+0.56+0.56,-0.45) to (0.28+0.56+0.56,0);\end{tikzpicture}
	=\begin{tikzpicture}[baseline = -1mm, color=\clr]
	\draw[-,thick ] (0.28+0.56+0.56,0) to[out=90, in=0] (0+0.56+0.56,0.3);
	\draw[-,thick ] (0+0.56+0.56,0.3) to[out = 180, in = 90] (-0.28+0.56+0.56,0);
	
\draw[-,thick ] (0.28+0.56+0.56+0.56,-0.45) to[out=120,in=-90] (0.28+0.56+0.56,0);
	\draw[-,thick ](0.28+0.56+0.56,-0.45) to[out=60,in=-90] (0.28+0.56+0.56+0.56,0);
    \draw[-,thick ] (0.28+0.56,-0.45) to (0.28+0.56,-0);
\end{tikzpicture}, \quad\begin{tikzpicture}[baseline = -9mm, color=\clr]
	
	
  \draw[-,thick ] (0.28+0.56,-0.9) to[out=120,in=-90] (0.28,-.45);
	\draw[-,thick ] (0.28,-0.9) to[out=60,in=-90] (0.28+0.56,-.45);
    \draw[-,thick ] (-0.28+0.56+0.56,-0.9) to[out=down,in=left] (0+0.56+0.56,-1.2) to[out=right,in=down] (0.28+0.56+0.56,-0.9);
    \draw[-,thick ] (-0.28+0.56+0.56+0.56,-0.45) to (-0.28+0.56+0.56+0.56,-0.9);
\end{tikzpicture}= \begin{tikzpicture}[baseline = -9mm, color=\clr]
	
	
  \draw[-,thick ] (0.28+0.56+0.56,-0.9) to[out=60,in=-90] (0.28+0.56+0.56+0.56,-.45);
	\draw[-,thick ] (0.28+0.56+0.56+0.56,-0.9) to[out=120,in=-90] (0.28+0.56+0.56,-.45);
    \draw[-,thick ] (-0.28+0.56+0.56,-0.9) to[out=down,in=left] (0+0.56+0.56,-1.2) to[out=right,in=down] (0.28+0.56+0.56,-0.9);
    \draw[-,thick ] (-0.28+0.56+0.56,-0.45) to (-0.28+0.56+0.56,-0.9);
\end{tikzpicture},
\end{equation}

   \begin{equation}\label{AB1}
                \begin{tikzpicture}[baseline = 7.5pt, scale=0.5, color=\clr]
            \draw[-,thick] (0,0) to[out=up, in=down] (1,2);
            \draw[-,thick] (0,2) to[out=up, in=down] (0,2.2);
            \draw[-,thick] (1,0) to[out=up, in=down] (0,2);
            \draw[-,thick] (1,2) to[out=up, in=down] (1,2.2);
             \draw(0,1.9)\bdot;
        \end{tikzpicture}
        ~-~
        \begin{tikzpicture}[baseline = 7.5pt, scale=0.5, color=\clr]
            \draw[-,thick] (0,0) to[out=up, in=down] (1,2);\draw[-,thick] (0,0) to[out=up, in=down] (0,-0.2);
             \draw[-,thick] (1,0) to[out=up, in=down] (0,2);\draw[-,thick] (1,0) to[out=up, in=down] (1,-0.2);
                        \draw(1,0.1)\bdot;
        \end{tikzpicture}
        ~=~
       \begin{tikzpicture}[baseline = 10pt, scale=0.5, color=\clr]
          \draw[-,thick] (2,2) to[out=down,in=right] (1.5,1.5) to[out=left,in=down] (1,2);
            \draw[-,thick] (2,0) to[out=up, in=right] (1.5,0.5) to[out=left,in=up] (1,0);
        \end{tikzpicture}
        ~-~
        \begin{tikzpicture}[baseline = 7.5pt, scale=0.5, color=\clr]
            \draw[-,thick] (0,0) to[out=up, in=down] (0,2);
            \draw[-,thick] (1,0) to[out=up, in=down] (1,2);
                   \end{tikzpicture}~,
    \end{equation}

  \begin{equation}\label{AB2}
               \begin{tikzpicture}[baseline = 10pt, scale=0.5, color=\clr]
            \draw[-,thick] (2,0) to[out=up, in=down] (2,1) to[out=up, in=right] (1.5,1.5) to[out=left,in=up] (1,1);
            \draw[-,thick] (1,1) to[out=down,in=right] (0.5,0.5) to[out=left,in=down] (0,1) to[out=up,in=down] (0,2);
           \draw( 1,1) \bdot;
        \end{tikzpicture}
        ~=~
        -\begin{tikzpicture}[baseline = 10pt, scale=0.5, color=\clr]
            \draw[-,thick] (0,0) to (0,2);
            \draw(0,1) \bdot;
                    \end{tikzpicture}
        ~=~
        \begin{tikzpicture}[baseline = 10pt, scale=0.5, color=\clr]
            \draw[-,thick] (2,2) to[out=down, in=up] (2,1) to[out=down, in=right] (1.5,0.5) to[out=left,in=down] (1,1);
            \draw[-,thick] (1,1) to[out=up,in=right] (0.5,1.5) to[out=left,in=up] (0,1) to[out=down,in=up] (0,0);
            \draw( 1,1) \bdot;
        \end{tikzpicture}~.
    \end{equation}

    In \cite[(1.21)]{RS3}, we have proved that the subcategory of $\AB$ whose objects  are $\mathbb N$ and whose morphisms are  generated by three generating  morphisms $\lcap,\lcup$ and $\swap$ is isomorphic to
     the Brauer category $\B$ over $\mathbb C$ in \cite[Definition~1.1]{RS3}. So, $\B$ can be identified with  a subcategory of $\AB$.
  For any $k\in\mathbb N$, define 
 \begin{equation}\label{defofdelta}
\Delta_k= \begin{tikzpicture}[baseline = 5pt, scale=0.5, color=\clr]
        \draw[-,thick] (0.6,1) to (0.5,1) to[out=left,in=up] (0,0.5)
                        to[out=down,in=left] (0.5,0)
                        to[out=right,in=down] (1,0.5)
                        to[out=up,in=right] (0.5,1);
        \draw (0,0.5) \bdot;
        \draw (-0.4,0.5) node{\footnotesize{$k$}};
    \end{tikzpicture}= \lcap\circ (\xdot~ \begin{tikzpicture}[baseline = 5pt, scale=0.5, color=\clr]
     \draw[-,thick] (0,0.15) to (0,1.15); \end{tikzpicture} ~)^k\circ \lcup\in \End_{\AB}(\ob 0).
\end{equation}
  Thanks to~\cite[Theorem~B]{RS3}, $\AB$  can be viewed as a $\mathbb C[\Delta_0, \Delta_2, \Delta_4, \ldots]$-linear category such that $$\Delta_kg:=\Delta_k\otimes g\text{  for any $g\in\Hom_{\AB}(\ob m,\ob s)$.}$$
  Suppose   $ \gamma=(\gamma_i)_{i\in \mathbb N}\in \mathbb C^\infty$ such that $\gamma$  is  admissible in the sense of \cite[Definition~2.10]{AMR}, i.e.,
$\gamma_j$'s satisfy the following conditions~\cite[(1.25)]{RS3}:
\begin{equation}\label{admomega}
2\gamma_k=-\gamma_{k-1}+\sum_{j=1}^{k}(-1)^{j-1}\gamma_{j-1}\gamma_{k-j}, \text{ for } k=1,3,\ldots
    \end{equation}
Specializing $\Delta_j$'s  at scalars $\gamma_j$'s, we have the (specialized) affine Brauer category \cite[(1.28)]{RS3}
 $$\AB(\gamma)=\mathbb C\otimes _{\mathbb C[\Delta_0, \Delta_2,\Delta_4,\ldots]}\AB$$ where $\mathbb C$ is the  $\mathbb C[\Delta_0,\Delta_2, \Delta_4,\ldots]$-module on which $\Delta_{2j}$ acts  as $\gamma_{2j}$ for any  $j\geq 0$.
  Let $\B(\gamma_0)$ be the subcategory of   $\AB(\gamma)$ whose objects are $\mathbb N$ and whose morphisms are  generated by $\lcup, \lcap$ and $\swap$. Then
 $$\B(\gamma_0)\cong \mathbb C\otimes_{\mathbb C[\Delta_0]} \B,$$ where $\mathbb C$ is considered as  $\mathbb C[\Delta_0]$-module on which  $\Delta_0$ acts as $\gamma_0$. Note that $\B(\gamma_0)$ is the Brauer category considered in \cite{LZ}.
Later on, we assume $\gamma_0=\delta$, where $\delta\in \mathbb C$. Let $A$ be the locally unital algebra associated to $\B(\delta)$. Then
 $$A=\bigoplus_{\ob a, \ob b\in \mathbb N} \text{Hom}_{\B(\delta)} (\ob a, \ob b).$$

It is known that to study representations of $\B(\delta)$   is equivalent to study the representations of the locally unital algebra $A$.
Later on, we identify $A$ with $\B(\delta)$ by abusing of notation.
We are going to follow the notations  from \cite{BS}. Let $\B(\delta)$-lfdmod be the category  of locally finite dimensional left $\B(\delta)$-modules. For each $\lambda\in \Lambda$ in \eqref{lambda},  we construct  a standard module $\Delta(\lambda)$ and a costandard module $\nabla(\lambda)$ in \cite[(2.18)]{RS2}. For the definition of a triangular category, see \cite[Definition~5.31]{BS} or \cite{SS}.

\begin{Theorem} \label{hd1}Let $\B(\delta)$ be the Brauer category over $\mathbb C$.
 \begin{itemize}
 \item[(1)] Each $ \Delta(\lambda)$ has the simple head $L(\lambda)$  for all $\lambda\in\Lambda$, and    $\{L(\lambda)\mid \lambda\in\Lambda\}$ is a complete set of pair-wise  non-isomorphic simple $\B(\delta)$-modules.
\item [(2)] $\B(\delta)$ is semi-simple if and only if $\delta\not\in \mathbb Z$.
\item [(3)] $\B(\delta)$-lfdmod is an upper finite highest weight category in the sense of \cite[Definition~3.42]{BS}.
 \end{itemize}
\end{Theorem}
\begin{proof} In fact, (1) follows from   \cite[Theorem~2.8]{RS2} and (2)
follows from   \cite[Theorem~3.9]{RS2}. Finally,    (3) is a  special case of \cite[Lemma~3.5]{BS} which says that
 the  $A$-lfdmod is an upper finite highest weight category if $A$ is the locally unital algebra associated to a triangular category such that its ``Cartan'' part is  semi-simple.\end{proof}

Two simple $\B(\delta)$-modules $L(\lambda)$ and $L(\mu)$ are said to be  in the same block if there is a sequence $\lambda^{(1)}=\lambda, \lambda^{(2)}, \ldots, \lambda^{(k)}=\mu$ in $\Lambda$  such that there is a
nontrivial extension between  $L(\lambda^{(i)})$ and $L(\lambda^{(i+1)})$ for $i=1,2,\ldots,k-1$.
In the current case (i.e. Theorem~\ref{hd1}(3)), it is equivalent to saying that there is a sequence $\lambda^{(1)}=\lambda, \lambda^{(2)}, \ldots, \lambda^{(k)}=\mu$ in $\Lambda$  such that either $L(\lambda^{(i)})$ is a composition factor of $\Delta(\lambda^{(i+1)})$ or  $L(\lambda^{(i+1)})$ is a composition factor of   $\Delta(\lambda^{(i)})$ for all $1\le i\le k-1$.
Thanks to Theorem~\ref{hd1}(2), we always assume $\delta\in \mathbb Z$. Otherwise, $\B(\delta)$ is semi-simple and each block contains a single simple $\B(\delta)$-module.

Let  $\B(\delta)$-lfdmod$^\Delta$ be the category of left $\B(\delta)$-modules on which each object admits a  finite filtration with sub-quotients isomorphic to
  standard modules.  Define $$ [K_0(\B(\delta)\text{-lfdmod}^\Delta)]=\mathbb C\otimes_{\mathbb Z}K_0(\B(\delta)\text{-lfdmod}^\Delta) ,$$ where  $K_0(\B(\delta)\text{-lfdmod}^\Delta)$ is  the  Grothendieck group of $\B(\delta)$-lfdmod$^\Delta$.

For any   $\lambda\in \Lambda$,  let   $\lambda^t$ be  the transpose partition  of $\lambda$ such that $\lambda_i^t=|\{j\mid \lambda_j\geq i\}|$. For example, $\lambda^t=(3, 1)$ if $\lambda= (2,1,1)$.
Thanks to \cite[Proposition~10.6]{BW} there is an $\mathfrak{sl}_\infty$-isomorphism
$${\bigwedge}_{\frac\delta2-1}^{\infty}\mathbb W\cong V(\omega_{\frac{\delta-1}2}).$$

In \cite{RS2}, we obtain  a $\mathfrak g_\delta$-action  on $[K_0(\B(\delta) \text{-\rm lfdmod}^{\Delta})]$ such that  the action of the generators $e_i+f_{-i}$'s
are given by $\tilde E_i$, where  $\tilde E_i$ is an endofunctor of $\B(\delta) \text{-\rm lfdmod}$ defined via  $\xdot$ (see \cite[Lemma~3.9]{RS2}). It is well-known that there is an  $\mathfrak{sl}_\infty$-action on the representation category of symmetric groups such that the generators $e_i$'s and $f_i$'s are given by certain $i$-restriction functors and $i$-induction functors $E_i$'s and $F_i$'s, respectively.
 Using the relations among $\tilde E_i, E_i$ and $F_{-i}$  in \cite[Lemma~3.9(2)]{RS2} yields  the $\mathfrak g_\delta$-action  on $[K_0(\B(\delta) \text{-\rm lfdmod}^{\Delta})]$.  Furthermore, we have the following result.
\begin{Theorem}\label{cateofg}\cite[Theorem~5.1]{RS2} Suppose $\delta\in \mathbb Z$. Let $\mathfrak g_\delta$ be the subalgebra of $\mathfrak{sl}_\infty$ in subsection~\ref{SP}.
 As $\mathfrak g_\delta$-modules,
$[K_0(\B(\delta) \text{-\rm lfdmod}^{\Delta})]\cong  \bigwedge_d^{\infty}\mathbb W$, where $d=\frac{\delta}2-1$.  The required isomorphism   sends  $[\Delta(\lambda^t)]$ to  $w_{\mathbf i_{d,\lambda}}$ for any $\lambda\in \Lambda$.
\end{Theorem}


For any $\lambda\in\Lambda$, we identify $\lambda$ with its associated  Young diagram. It is the diagram on which there are $\lambda_i$ boxes in the $i$th row. For any box $x$ in row $i$ and column $j$ of the Young diagram of $\lambda$, let \begin{equation}\label{con}c_{\delta}(x)=\frac{\delta-1}{2}+c(x)\in I,\end{equation} where $I$ is given in \eqref{io} and  $c(x)= j-i$,  the usual  content of $x$.  Recall simple roots $\alpha_{i}$'s and fundamental weights $\omega_i$'s in subsection~\ref{lie}.
In \cite[Definition~4.5]{RS2}, we define the weight function $\text{wt}: \Lambda \rightarrow P$ such that  \begin{equation}\label{wt123}
\text{wt}(\lambda)=\omega_{\frac{\delta-1}{2}}-\sum_{x\in \lambda}\alpha_{c_{\delta}(x)}\in P. \end{equation}
It induces the weight function $\bar{\text{wt}}: \Lambda\rightarrow P_\theta $.

\begin{Lemma}\label{key1}Suppose $d=\frac{\delta}2-1$.  For any $\lambda\in \Lambda$, $\text{wt}(\lambda)=\text{wt}_d(\mathbf i_{d, \lambda^t})$,   where $\mathbf i_{d, \lambda}$  is given in \eqref{ilambda} and $\text{wt}_d$ is given in Definition~\ref{phi} .
\end{Lemma}
\begin{proof} Recall that  $\bigwedge_d^{\infty}\mathbb V$    has basis   $\{v_{\lambda,d}\mid \lambda\in \Lambda\}$ \cite{BW}, where $$v_{\lambda,d}=v_{\lambda_1+d}\wedge v_{\lambda_2+d-1}\wedge v_{\lambda_3+d-2}\wedge\ldots.$$
Thanks to \cite[Proposition~10.6]{BW}, \begin{equation}\label{isso} {\bigwedge}_d^{\infty}\mathbb V\cong {\bigwedge}_d^{\infty}\mathbb W\end{equation} as $\mathfrak {sl}_\infty$-modules and the required isomorphism sends
$v_{\lambda,d} $ to  $w_{\mathbf i_{d, \lambda^t}}$, where $w_{\mathbf i_{d, \lambda^t}}$ is given in \eqref{wbi} and \eqref{ilambda} for any $\lambda\in \Lambda$.
Since we are assuming  $d=\frac{\delta}2-1$, one can check that the $\mathfrak {sl}_\infty$-weight of   $v_{\lambda,d} $ is $\text{wt}(\lambda)$.
  Thanks to \eqref{isso}, $\text{wt}(\lambda)=\text{wt}_d(w_{\mathbf i_{d, \lambda^t}})=\text{wt}_d(\mathbf i_{d, \lambda^t})$, proving the result.
\end{proof}

In \cite[Corollary~4.7]{RS2}, we have  proved that  $\text{wt}(\lambda)\equiv \text{wt}(\mu)\pmod{Q^\theta}$ if
$L(\lambda)$ and $L(\mu)$ are in the same block. So, $\bar{\text{wt}}^{-1}( \gamma)$ is a union of certain blocks of $\B(\delta)$-lfdmod for
any  $\gamma=\bar{\text{wt}(\lambda)}\in P_\theta$ and $\lambda\in \Lambda$.
The following result gives the explicit decomposition of $\bar{\text{wt}}^{-1}( \gamma)$   into blocks of $\B(\delta)$-lfdmod.

\begin{Theorem}\label{main} Suppose  $\lambda, \mu \in \Lambda$ and $d=\frac\delta2-1$.
\begin{enumerate}\item  Simple $\B(\delta)$-modules $L(\lambda^t)$ and $L(\mu^t)$ are in the same block if and only if there is a $w\in W_\infty$ such that $w(\mathbf i_{d, \lambda})=\mathbf i_{d, \mu}$.
\item Suppose  $\gamma=\bar{\text{wt}}(\lambda)$ and  $\mathbf i_{d, \lambda^\t}=(i_1, i_2, \ldots)$.
\begin{itemize}\item[(1)] $\bar {\text{wt}}^{-1} (\gamma)$ is a single block of $\B(\delta)$-lfdmod if  either $\delta$ is odd or $\delta$ is even and  $i_k=0$ for some $k\ge 1$. \item [(2)] $\bar {\text{wt}}^{-1} (\gamma)$    is a disjoint union of  two different blocks of  $\B(\delta)$-lfdmod if
    $\delta$ is even and $i_k\neq 0$ for all $k\ge 1$. More explicitly, one block contains $\lambda$ and the another block contains $\mu$ such that $ \mathbf i_{d, \mu^t}=(-i_k, i_1, \cdots, i_{k-1}, i_{k+1}, \ldots)$ for some  $k$ sufficiently large. \end{itemize}\end{enumerate}
   In any case, each block contains an infinite number of simple modules and all blocks of $\B(\delta)$-lfdmod can be obtained in this way.
\end{Theorem}
\begin{proof} Recall that $A$ is the locally unital algebra associated to $\B(\delta)$.  It is explained in  \cite[Lemma~2.8]{LZ} that  the centralizer subalgebra $1_{\ob n} A 1_{\ob n}$ of $A$ is isomorphic to the Brauer algebra $B_n(\delta)$ in \cite{B}. Graham and Lehrer~\cite{GL} showed that
 $B_n(\delta)$ is a cellular algebra (over an arbitrary field). Since Brauer category $\B(\delta)$ is an upper finite  triangular category in the sense of \cite[Definition~5.31]{BS} and an upper finite triangular category is an upper finite weakly triangular category in the sense of \cite[Definition~2.1]{GRS}, we can use~\cite[Proposition~6.10]{GRS}. This implies that  two simple $A$-modules $L(\lambda^t)$ and $L(\mu^t)$ are in the same block if and only if there is an $\ob n$ such that simple $1_{\ob n} A 1_{\ob n}$-modules  $1_{\ob n} L (\lambda^t)$ and $1_{\ob n} L (\mu^t)$ are in the same block.  In the current case, $1_{\ob n} L (\lambda^t)$ and $1_{\ob n} L (\mu^t)$ are  in fact the simple $B_n(\delta)$-modules $L(\lambda^\t)$ and $L(\mu^t)$ in \cite[Theorem~4.1]{CVM}. This can be checked by comparing the construction of the Specht modules for Brauer algebras in \cite{CVM0} and our construction of standard modules $\Delta(\lambda)$'s in \cite{RS2} and using \cite[Proposition~4.3]{GRS}. The last one establishes an explicit relationship between standard modules of $A$ and Specht modules of $B_n(\delta)$. In \cite[Theorem~4.2]{CVM} Cox, De Visscher and Martin have proved that $1_{\ob n} L (\lambda^t)$ and $1_{\ob n} L (\mu^t)$
are in the same block of $B_n(\delta)$  if and only if $\lambda\in W_n.\mu$, where the dot action is defined as
$$w.\lambda=w(\lambda+\rho_n)-\rho_n$$ and   $\rho_n=\sum_{i=1}^n (1-i-\frac{\delta}{2}) \delta_i$. Here, we identify $\lambda\in \Lambda$ with $\sum_{i=1}^\infty \lambda_i \delta_i$. Obviously, $w.\lambda=\mu$ for $w\in W_n$ and $\lambda, \mu\in \Lambda_n$ if and only if $w.\lambda=\mu$ if we consider $w\in W_\infty$ and $\lambda, \mu\in \Lambda_\infty$ by assuming $\lambda_j=\mu_j=0$ for all $j\ge n+1$. In this case,
$$w.\lambda=w(\lambda+\rho_\infty)-\rho_\infty$$
and $\rho_\infty=\sum_{i=1}^\infty (1-i-\frac{\delta}{2})\delta_i$.  This is equivalent to saying that $\mathbf i_{d,\mu}=w(\mathbf i_{d,\lambda})$, proving (a). Now, (b)  follows from (a), Lemma~\ref{key1} and  Theorem~\ref{weyl} for $d=\frac{\delta}2-1$.
 \end{proof}

Theorem~\ref{main} gives an explicit decomposition of ``block" in \cite[Lemma~4.9(2)]{RS2}. It can be
 generalized properly for cyclotomic Kauffman categories, cyclotomic oriented Brauer categories and cyclotomic Brauer categories. Details will be given elsewhere.

\section{``Central blocks" of  the Brauer category over $\mathbb C$}

Recently, Brundan and Vargas use a subalgebra of the center of the partition category to classify its block~\cite{BW}.  Following their approach,  we define a ``central block" of Brauer category $\B(\delta)$ in Definition~\ref{sim}. Finally, we give an example to illustrate that this "central block" may be bigger than the  $\mfg_\delta$-weight space of the $(\frac\delta2-1)$-th sector of semi-infinite wedge space.
   See Remark~\ref{fre}.

Let  $AB$ be the locally unital algebra   associated to the affine Brauer category $\AB$, i.e.,
$$AB=\bigoplus_{\ob a,\ob b}\Hom_{\AB}(\ob a,\ob b) \text { such that } 1_\ob b AB1_\ob a =\Hom_{\AB}(\ob a,\ob b).$$
For any $n\in \mathbb N$,  define   $AB_n=1_\ob n AB1_\ob n$. It is proved in \cite[Theorem~3.14]{RS3} that $AB_n$ is isomorphic to the affine Nazarov-Wenzl algebra $W_n^{\rm aff}$  over $\mathbb C$ (called  affine Wenzl algebra previously  in \cite{Na}).
 Let   \begin{equation}\label{ou}  O(u)=u-{\frac{1}{2}}+\sum_{a=0}^\infty
\frac{\Delta_a} {u^{a}}\in AB_0[[u^{-1}]],\end{equation}
where $\Delta_a$ is defined in \eqref{defofdelta}.

\begin{Lemma} \label{O} $O(u)O(-u)=(\frac{1}{2}-u)(\frac{1}{2}+u)$.\end{Lemma}
\begin{proof} Let $\phi=O(u)O(-u)-(\frac{1}{2}-u)(\frac{1}{2}+u)$. By~\eqref{ou}, we have
$$\begin{aligned} \phi
& =(u-\frac{1}{2})\sum_{a=0}^\infty
\frac{\Delta_a} {(-u)^{a}}+\sum_{a=0}^\infty
\frac{\Delta_a} {u^{a}}(-u-\frac{1}{2})+\sum_{a=0}^\infty
\frac{\Delta_a} {u^{a}}\sum_{a=0}^\infty
\frac{\Delta_a} {(-u)^{a}}
\\&=\sum_{i\geq -1}(((-1)^{i+1}-1)\Delta_{i+1}-\frac{1}{2}((-1)^i+1)\Delta_{i}+\sum_{j=0}^i(-1)^{i-j}\Delta_i\Delta_{i-j})u^{-i}.\end{aligned}$$
We prove $\phi=0$ by showing that any coefficient $b_i$ of $u^{-i}$ in  $\phi$ is zero.
Obviously,   $b_{i}=0$ for any $i\le -1$. We have $b_{2k}=0$ since $2\Delta_{2k+1}=-\Delta_{2k}+\sum_{j=1}^{2k+1} (-1)^{j-1} \Delta_{j-1}\Delta_{2k+1-j}$ for any non-negative integer $k$ (see \cite[Lemma~3.4]{RS3}).
Finally,    $$b_{2k+1}=\sum_{j=0}^{2k+1}(-1)^{2k+1-j}\Delta_{j}\Delta_{2k+1-j}
=\sum_{j=0}^k\Delta_j\Delta_{2k+1-j}((-1)^{2k+1-j}+(-1)^{j})=0.$$
\end{proof}

Let $\begin{tikzpicture}[baseline = 5pt, scale=0.5, color=\clr]
     \draw[-,thick] (0,-0.15) to (0,1.15); \draw (0,0.5) \bdot;  \draw (0.5,0.5) node{\footnotesize{$x^r$}};\end{tikzpicture} = \begin{tikzpicture}[baseline = 5pt, scale=0.5, color=\clr]
     \draw[-,thick] (0,-0.15) to (0,1.15); \draw (0,0.5) \bdot;  \draw (0.4,0.5) node{\footnotesize{$r$}}; \end{tikzpicture}$. So, $f(\begin{tikzpicture}[baseline = 5pt, scale=0.5, color=\clr]
     \draw[-,thick] (0,-0.15) to (0,1.15); \draw (0,0.5) \bdot;  \end{tikzpicture})=\begin{tikzpicture}[baseline = 5pt, scale=0.5, color=\clr]
     \draw[-,thick] (0,-0.15) to (0,1.15); \draw (0,0.5) \bdot;  \draw (0.75,0.5) node{\footnotesize{$f(x)$}}; \end{tikzpicture}$
     for any $f(x)\in \mathbb C[[x]]$.

      \begin{Lemma}  \label{change} We have
      $$ \begin{tikzpicture}[baseline = 5pt, scale=0.5, color=\clr]
     \draw[-,thick] (0,0.15) to (0,1.15); \end{tikzpicture}\text{ }  O(u)=O(u) \text{ }
     \begin{tikzpicture}[baseline = 10pt, scale=0.5, color=\clr]
            \draw[-,thick] (0,0) to (0,2);
            \draw(0,1) \bdot;  \draw (1,1) node{\footnotesize{$\gamma_x(u)$}};
                    \end{tikzpicture}$$
                      where $ \gamma_x(u)=\frac{(u+x)^2-1}{(u-x)^2-1}\frac{(u-x)^2}{(u+x)^2}$, and $\gamma_x(u)$  has to be interpreted as a formal Laurent series in $\mathbb C [x][[u^{-1}]]$.
                       \end{Lemma}
\begin{proof} Thanks to \cite[Theorem~3.14]{RS3},  $W_n^{\rm aff}\cong AB_n$ as $\mathbb C$-algebras. When $\Delta_0$ is specialized at any complex number $N$, $W_n^{\rm aff}$ is the affine Wenzl algebra in \cite{Na}. In this case,
the required isomorphism sends  the diagram $1_{\ob k-\ob 1}\otimes O(u) \otimes 1_{\ob n-\ob k-\ob 1}$
 to  $W_{k}(u)+u-\frac12$ in \cite[(4.8)]{Na}. Therefore, Lemma~\ref{change} follows from \cite[(4.8)]{Na} when $\Delta_0$ is specialized at an arbitrary  complex number $N$. So Lemma~\ref{change}  holds when $\Delta_01_{\ob n}$ is a central element in $W_n^{\rm aff}$.\footnote{Nazarov's arguments on the proof of \cite[(4.8)]{Na} are still available if one assume that $N$ is a central element.}\end{proof}

    \begin{Lemma} Let   $c_n(u):=O(-u)\otimes 1_{\ob n}\otimes O(u)\in AB_n[[u^{-1}]]$. Then  $$c_n(u) = (\frac{1}{2}-u)(\frac{1}{2}+u) \prod_{i=1}^n\frac{(u+x_i)^2-1}{(u-x_i)^2-1}\frac{(u-x_i)^2}{(u+x_i)^2},$$
where $ x_i = \begin{tikzpicture}[baseline = 12.5, scale=0.35, color=\clr]
       \draw[-,thick](-0.4,1.1)to[out= down,in=up](-0.4,2.5);
       \draw(0.5,1.1) node{$ \cdots$}; \draw(0.5,2.5) node{$ \cdots$};
       \draw[-,thick](1.8,1.1)to (1.8,2.5);
       \draw(-0.4,3)node{\tiny$1$};
       \draw(3,3)node{\tiny$i$}; \draw(5.5,3)node{\tiny$n$};
        \draw[-,thick] (3,1) to[out=up, in=down] (3,2.5);
         \draw[-,thick](3.6,1.1)to[out= down,in=up](3.6,2.5);\draw (3,1.8)\bdot;
         \draw(4.5,1.1) node{$ \cdots$}; \draw(4.5,2.5) node{$ \cdots$};
         \draw[-,thick](5.5,1.1)to[out= down,in=up](5.5,2.5);
     \end{tikzpicture}$.
\end{Lemma}
\begin{proof}The following result follows immediately from Lemmas~\ref{O}--\ref{change}.\end{proof}

\begin{Lemma}\label{sca}$C(u)=(c_n(u))_{n\in \mathbb N} \in Z(AB)[[u^{-1}]]$ where  $$Z(AB):=\{(z_n)_{n\in \mathbb N}\in\prod_{n=0}^\infty AB_n \mid z_m a=a z_n, \forall a\in 1_{\ob m} AB 1_{\ob n}, m, n\in \mathbb N\} .$$

\end{Lemma}
\begin{proof} Suppose $f\in 1_{\ob m} AB1_{\ob t}$. Then
 $$\begin{tikzpicture}[baseline = 3mm, color=\clr]
                \draw[-,thick] (0, 0.95)to[out=up,in=down](0,0.95+0.6); \draw(0.5, 0.88+0.4) node{$ \cdots$};   \draw[-,thick] (1,0.95)to[out=up,in=down](1,0.95+0.6);
                \draw (-0.1,0.4) rectangle (1.2,0);
                \draw(0.5, 0.2) node{$f$};
                \draw (1.2,0.55) rectangle (-0.1,0.95); \draw[-,thick] (0, 0.4)to[out=up,in=down](0,0.55); \draw(0.5, 0.48) node{$ \cdots$};   \draw[-,thick] (1,0.4)to[out=up,in=down](1,0.55);
                \draw(0.5, 0.7) node{$c_m(u)$}; 
                \draw[-,thick] (0.2, 0)to[out=up,in=down](0.2,-0.55); \draw(0.5, -0.2) node{$ \cdots$};   \draw[-,thick] (0.8,0)to[out=up,in=down](0.8,-0.55);
    \end{tikzpicture}=
    \begin{tikzpicture}[baseline = 3mm, color=\clr]
                \draw (-0.1,0.4) rectangle (1.2,0);
                \draw(0.5, 0.2) node{$f$};
                \draw(-0.8, 0.8) node{$O(-u)$};
                 \draw(1.5, 0.8) node{$O(u)$};
               \draw[-,thick] (0, 0.4)to[out=up,in=down](0,1.05); \draw(0.5, 0.68) node{$ \cdots$};   \draw[-,thick] (1,0.4)to[out=up,in=down](1,1.05);
                \draw[-,thick] (0.2, 0)to[out=up,in=down](0.2,-0.35); \draw(0.5, -0.2) node{$ \cdots$};   \draw[-,thick] (0.8,0)to[out=up,in=down](0.8,-0.35);
    \end{tikzpicture}=
    \begin{tikzpicture}[baseline = 3mm, color=\clr]
                \draw (-0.1,0.4) rectangle (1.2,0);
                \draw(0.5, 0.2) node{$f$};
               \draw[-,thick] (0, 0.4)to[out=up,in=down](0,1.05); \draw(0.5, 0.68) node{$ \cdots$};   \draw[-,thick] (1,0.4)to[out=up,in=down](1,1.05);
                \draw[-,thick] (0.2, 0)to[out=up,in=down](0.2,-0.65); \draw(0.5, -0.2) node{$ \cdots$};   \draw[-,thick] (0.8,0)to[out=up,in=down](0.8,-0.65);
                \draw(-0.8, -0.5) node{$O(-u)$};
                 \draw(1.5, -0.5) node{$O(u)$};
    \end{tikzpicture}=\begin{tikzpicture}[baseline = 3mm, color=\clr]
                \draw[-,thick] (0, 0.95)to[out=up,in=down](0,0.95+0.6); \draw(0.5, 0.88+0.4) node{$ \cdots$};   \draw[-,thick] (1,0.95)to[out=up,in=down](1,0.95+0.6);
                \draw (-0.1,0.4) rectangle (1.2,0);
                \draw(0.5, 0.2) node{$c_t(u)$};
                \draw (1.2,0.55) rectangle (-0.1,0.95); \draw[-,thick] (0.2, 0.4)to[out=up,in=down](0.2,0.55); \draw(0.5, 0.48) node{$ \cdots$};   \draw[-,thick] (0.8,0.4)to[out=up,in=down](0.8,0.55);
                \draw(0.5, 0.7) node{$f$}; 
                \draw[-,thick] (0.2, 0)to[out=up,in=down](0.2,-0.55); \draw(0.5, -0.2) node{$ \cdots$};   \draw[-,thick] (0.8,0)to[out=up,in=down](0.8,-0.55);
    \end{tikzpicture},
$$ and  the result follows.
\end{proof}
By the  arguments after \cite[Assumption 2.1]{RS3}, $\B(\delta)$ is the quotient category of $\AB$ with respect to the polynomial $f(x)=x-\delta$ and $\gamma=\{\delta(\frac{\delta-1}{2})^i\mid i\in \mathbb N\}$.  So, a $\B(\delta)$-module $M$ can be regarded as an $\AB$-module. Suppose $\lambda\in \Lambda_n$. In \cite[Remark~3.13]{RS2}  we have explained  that  $1_n\Delta(\lambda)$ is the cell module of Brauer algebra $B_n(\delta)$ in \cite{B}. In fact, it is isomorphic to the Specht module $S^\lambda$ for the symmetric group $\mathfrak S_n$, and hence    $1_n\Delta(\lambda)$ is an  irreducible $B_n(\delta)$-module. Now, we use well-known results on the orthogonal form of a Specht module for the symmetric group in the semisimple case~\cite{GJ}.  $1_n\Delta(\lambda)$ has an orthogonal basis, say $v_{\mathbf t}$, where $\mathbf t$ ranges over all standard $\lambda$-tableaux. Recall that a standard $\lambda$-tableau is obtained from the Young diagram of the partition $\lambda\in \Lambda_n$ by inputting numbers $1, 2, \ldots, n$ into the boxes such that the entries are increasing from left to right down columns.
 So, $\mathbf t$ corresponds to a unique up-tableaux $(\mathbf t_0, \mathbf t_1, \ldots, \mathbf t_n)$ such that $\mathbf t_i$ is obtained from $\mathbf t$ by removing all entries strictly bigger than $ i$.
 Of course, $\mathbf t_i$ corresponds to a  partition of $i$.
  Let $y=\mathbf t_{i}\setminus \mathbf t_{i-1}$. In other words, $y\cup \mathbf t_{i-1}=\mathbf t_i$. Then  $x_i$ acts on $v_{\mathbf t}$ as the scalar $ c_\delta(y)$ defined in \eqref{con}.  By Lemma~\ref{sca}, $c_n(u)$ acts on $1_n\Delta(\lambda)$  as the scalar
\begin{equation}\label{con2}C(u)|_{\Delta(\lambda)}:=c_n(u)|_{1_n\Delta(\lambda)}=(\frac{1}{2}-u)(\frac{1}{2}+u)\prod_{y\in \lambda}\frac{(u+c_\delta(y))^2-1}{(u-c_\delta(y))^2-1}\frac{(u-c_\delta(y))^2}{(u+c_\delta(y))^2}\in \mathbb{C}[[u]].\end{equation}

Motivated by \cite[(5.7)]{BV} on partition categories, we give the following definition.
\begin{Defn}\label{sim}For any $\lambda, \mu\in\Lambda$, we say $\lambda\sim\mu$ if $C(u)|_{\Delta(\lambda)}=C(u)|_{\Delta(\mu)}$.\end{Defn}
The equivalence classes of $\Lambda$ with respect to $\sim$ are called ``central blocks".  It is defined via the subalgebra of the center  generated by
$c^{(r)}$, $r\in \mathbb N$, where  $\sum_{r=0}^\infty c^{(r)} u^{-r}:=(c_n(u))_{n\in \mathbb N}$.

\begin{Prop}\label{cen}Suppose $\lambda, \mu\in\Lambda$. \begin{enumerate} \item If $\bar{\text{wt}}(\lambda )=\bar{\text{wt}}(\mu) $, then  $\lambda\sim\mu$.
 \item If $\delta$ is even and $\lambda\sim \mu$, then
  $\bar{\text{wt}}(\lambda)=\bar{ \text{wt}}( \mu) $.\end{enumerate}
\end{Prop}
\begin{proof}Suppose $(\lambda, \mu)\in \Lambda_n\times \Lambda_m$ and  $n\geq m$. If $\bar{\text{wt}}(\lambda )=\bar{\text{wt}}(\mu) $ , by   \eqref{wt123}, we have  $n=m+2s$ for some $s\geq 0$. Furthermore,  up to a permutation,  the sequence $(c_\delta(y))_{y\in \mu}$ can be obtained from $(c_\delta(y))_{y\in \lambda}$ by removing  some $\pm i_j\in I$, $1\leq j\leq s$. By \eqref{con2}, $C(u)|_{\Delta(\lambda)}=C(u)|_{\Delta(\mu)}$, and   (a) follows.

  For any rational function $f(u)$, define
$$\text{wt }(f(u)):=\sum_{a\in \mathbb C}(z_a-p_a)\omega_a,$$
where $z_a$ (resp., $p_a$) is the multiplicity of $a$ as a zero (resp., pole) of $f(u)$. We call $\text{wt }(f(u))$  the weight of $f(u)$.
So, $\text{wt}(\gamma_a(u))=\alpha_a-\alpha_{-a}$, where $\gamma_a(u)$ is given in Lemma~ \ref{change}.
 Suppose  $\lambda\sim\mu$. By \eqref{con2},
 \begin{equation}\label{djei}\prod_{y\in \lambda}\frac{(u+c_\delta(y))^2-1}{(u-c_\delta(y))^2-1}\frac{(u-c_\delta(y))^2}{(u+c_\delta(y))^2}=\prod_{y\in \mu}\frac{(u+c_\delta(y))^2-1}{(u-c_\delta(y))^2-1}\frac{(u-c_\delta(y))^2}{(u+c_\delta(y))^2}.
 \end{equation}
 The weight of the LHS of  \eqref{djei} is $\sum_{y\in \lambda}\alpha_{c_\delta(y)}-\sum_{y\in\lambda}\alpha_{-c_\delta(y)}$, while the weight of the RHS
 is $\sum_{y\in \mu}\alpha_{c_\delta(y)}-\sum_{y\in\mu}\alpha_{-c_\delta(y)}$. This proves that
    \begin{equation}\label{kkkl} 2\sum_{y\in \lambda}\alpha_{c_\delta(y)}-2\sum_{y\in \mu}\alpha_{c_\delta(y)}=
    \sum_{i\in I} b_i (\alpha_i+\alpha_{-i})\end{equation}
    for some $b_i\in \mathbb Z$.
    When $\delta$ is even, $I=\frac12+\mathbb Z$. In this case,  all $b_i$'s  in \eqref{kkkl} has to  be even, forcing
    $$\sum_{y\in \lambda}\alpha_{c_\delta(y)}\equiv \sum_{y\in \mu}\alpha_{c_\delta(y)} \text{ (mod $Q^\theta$)}.$$
     Thanks to  \eqref{wt123}, $\bar{\text{wt}}(\lambda )=\bar{\text{wt}}(\mu) $, proving  (b).
\end{proof}
Thanks to Proposition~\ref{cen}, each ``central block" corresponds to a  $\mathfrak g_\delta$-weight space of $(\frac \delta2 -1)$th semi-infinite wedge space if $\delta$ is even.
In the remaining case, a ``central block" may be bigger than
  a $\mathfrak g_\delta$-weight space of $(\frac \delta2 -1)$th semi-infinite wedge space. This can be seen by the following example.

\begin{example}\label{fre} Suppose   $\delta=1$, $\lambda=(2, 2)\in \Lambda_4$ and $\lambda=(2, 1)\in \Lambda_3$. By \eqref{con2}, $$C(u)|_{\Delta(\lambda)}=C(u)|_{\Delta(\mu)}=(\frac{1}{2}-u)(\frac{1}{2}+u),$$ and hence $\lambda\sim \mu$. However,  by \eqref{wt123}, $\text{wt}(\lambda)-\text{wt}(\mu)=-\alpha_0\notin Q^\theta$.
\end{example}
~~~~\\
\noindent\textbf{Data Availability}~~\quad ~Data sharing not applicable to this article; calculations follow from results recorded in the
text of the article.

\noindent\textbf{Ethical Statement/Conflict of Interest}~\quad~~ We certify that this manuscript is original and has not been published and will not be submitted elsewhere for publication while being considered by Algebras and Representation Theory.  And the study is not split up into several parts to increase the quantity of submissions and submitted to various journals or to one journal over time. No data have been fabricated or manipulated (including images) to support our conclusions. No data, text, or theories by others are presented as if they were our own.
The submission has been received explicitly from all co-authors. And authors whose names appear on the submission have contributed sufficiently to the scientific work and therefore share collective responsibility and accountability for the results. The authors declare that they have no conflict of interest.
This article does not contain any studies with human participants or animals performed by any of the authors.
Informed consent was obtained from all individual participants included in the study.


\begin{thebibliography}{DWH99}
\bibitem{Ari} {\scshape S.~ Ariki}, {\og On the decomposition numbers of the Hecke algebra of
$G(m,1,n) $\fg }, \emph{J. Math. Kyoto Univ.},
\textbf{36}, Number 4 (1996), 789--808.
\bibitem{AMR}{\scshape S.~Ariki, A.~Mathas {\normalfont \smfandname}  H. Rui},
{\og {Cyclotomic Nazarov-Wenzl algebras}\fg}, \emph{Nagoya Math.J.}, Special issue in honor of Prof. G. Lusztig's sixty birthday,
  \textbf{182} (2006), 47--134.
\bibitem{BW}{\scshape H.Bao and  W. Wang},
{\og A new approach to Kanzhdan-Lusztig theory of type $B$ via quantum symmetric pairs\fg}, \emph{Asterisque}, \textbf{402}(2018), vii+134.




\bibitem{B}{\scshape R.~Brauer}, {\og On algebras which are connected with the semisimple
  continuous groups\fg}, \emph{Ann. of Math.} \textbf{38} (1937), 857--872.

       \bibitem{BK1} {\scshape J. Brundan, A. Kleshchev}, {\og Graded decomposition numbers for cyclotomic Hecke algebras\fg}, \emph{Adv. in Math.},   \textbf{222}, (2009), 1883--1942.

\bibitem{BS}{\scshape J.Brundan, C. Stoppel}, {\og Semi-infinite highest weight categories \fg}, ArXiv:1808.08022v4,
to appear in Memoir of AMS.

\bibitem{BV}{\scshape J.Brundan,    M. Vargas}, {\og A new approach to the representation theory of the partition category \fg},
arXiv:2107.05099v1 [math.RT].

\bibitem{CVM0} {\scshape A.~Cox,  M.~ De Visscher, P.Martin},
{\og The  blocks of the  Brauer algebra in characteristic zero\fg}, \emph{Represent. Theory},
\textbf{ 13},  (2009),  272--308.
\bibitem{CVM} {\scshape A.~Cox,  M.~ De Visscher and P.Martin},
{\og The  geometric characterisation of the blocks of the  Brauer algebra\fg}, \emph{J. London Math. Soc.}(2),
\textbf{80},  (2009),  471-494.



\bibitem{ES} {\scshape M. Ehrig and C. Stroppel}, {\og Nazarov-Wenzl algebras, coideal subalgberas and categorified skew Howe duality \fg},
   \emph{ Adv. Math.}, \textbf {331}   (2018) 58--142.

    \bibitem{GRS}{\scshape M. Gao, H. Rui {\normalfont \smfandname} L. Song}, {\og Representaions of weakly triangular categories\fg}, arXiv:2012.02945 [math. RT].
\bibitem{GL} {\scshape J.~J. Graham {\normalfont \smfandname} G.~I. Lehrer}, {\og Cellular  algebras\fg}, \emph{Invent. Math.} \textbf{123} (1996), 1--34.
\bibitem{GJ} {\scshape G. James}, {\og The representation theory of the symmetric groups\fg}, \emph{Proc. Sympos.  Pure Math.} \textbf{47} (1987), 111--126.
\bibitem{KW}{\scshape V.G.Kac, S.P.Wang}, {\og On automprphisms of Kac-Moody algebras and groups\fg}, \emph{Adv.Math.}, \textbf{92} (1992),129--195.
\bibitem{Kle}{\scshape A. Kleshchev}, {\og Representation theory of symmetric groups and its related Hecke algebras\fg}, \emph{Bull. Amer. Math. Soc.} (N.S.), \textbf{47} (2010),  no.~3, 419--481.

\bibitem{LLT}{\scshape A. Lascoux, B. Leclerc and J.-Y. Thibon}, {\og Hecke algebras at roots of unity and crystal bases of quantum affine algebras\fg}, \emph{Comm. Math. Phys.}, \textbf{181} (1996), 205-263.


\bibitem{LZ} {\scshape G. Lehrer and R.B. Zhang}, {\og  The Brauer category and invariant Theory\fg},  \emph{J.  Eur. Math. Soc.}, \textbf{17} (2015), no. 9, 2311--2351
 \bibitem{Na}{\scshape M. Nazarov}, {\og Young's orthogonal form for Brauer's centralizer algebra\fg}, \emph{J.Algebra}
\textbf{182}, (1996), 664--693.





 \bibitem{RS2}{\scshape H. Rui and  L. Song}, {\og  Representation of Brauer category and categorification\fg}, \emph{J. Algebra}, \textbf{557},  (2020), 1--36.


 \bibitem{RS3}{\scshape H. Rui and  L. Song}, {\og  Affine Brauer category and parabolic category $\mathcal O$ in types $B,C,D$ \fg}, \emph{Math.Z}, \textbf{293} (2019), no. 1-2, 503--550

 \bibitem{SS}{\scshape S.Sam {\normalfont \smfandname} A. Snowden}, {\og  The representation theory of  Brauer categories, I, Triangular categories \fg}, \emph{preprint},  2020.


\end{thebibliography}
\end{document}